\documentclass[11pt,a4paper]{amsart}
\usepackage{amsfonts}
\usepackage{amsthm}
\usepackage{amsmath}
\usepackage{amscd}
\usepackage[latin2]{inputenc}
\usepackage{t1enc}
\usepackage[mathscr]{eucal}
\usepackage{indentfirst}
\usepackage{graphicx}
\usepackage{graphics}
\usepackage{pict2e}
\usepackage{epic}
\numberwithin{equation}{section}
\usepackage[margin=2.7cm]{geometry}
\usepackage{epstopdf} 
\usepackage[table]{xcolor}
\usepackage{longtable}
\usepackage{changepage}
\usepackage{microtype}
\usepackage{float}
\usepackage[noadjust]{cite}
\usepackage{adjustbox}
\usepackage{array}
\usepackage{multirow}
\usepackage{lscape}

\theoremstyle{plain}
\newtheorem{theorem}{Theorem}[section]

\newtheorem{proposition}[theorem]{Proposition}

\newtheorem{lemma}[theorem]{Lemma}
\theoremstyle{definition}
\newtheorem{definition}{Definition}[section]

\newtheorem{remark} [theorem] {Remark}
\newtheorem{condition}{Condition}[section]

\newcommand{\C}{\mathbb{C}}
\newcommand{\Z}{\mathbb{Z}}
\newcommand{\Q}{\mathbb{Q}}
\newcommand{\A}{\mathbb{A}}

\newcommand{\PP}{\mathbb{P}}
\newcommand{\Proj}{\mathbb{P}}

\DeclareMathOperator{\Pf}{Pf}
\DeclareMathOperator{\codim}{codim}
\DeclareMathOperator{\Imago}{Im}

\DeclareMathOperator{\rk}{rk}

\DeclareMathOperator{\Fix}{Fix}
\DeclareMathOperator{\Cl}{Cl}
\DeclareMathOperator{\Spec}{Spec}
\DeclareMathOperator{\proj}{Proj}
\DeclareMathOperator{\wts}{wts}
\DeclareMathOperator{\wt}{wt}
\DeclareMathOperator{\Bs}{Bs}

\usepackage{changepage}

\usepackage{makecell}
\usepackage[all]{xy}

\allowdisplaybreaks
\usepackage{mathpazo}
\usepackage{fancyhdr}
\usepackage{pgf,tikz}
\usetikzlibrary{arrows}
\usepackage{bm}
\usepackage{tikz}
\usepackage{lscape}
\usepackage{mathtools}

\usepackage[all]{xy}

\usepackage[colorlinks]{hyperref}

\usepackage{afterpage}
\usepackage{enumitem}
\usepackage{comment}
\usepackage{accents}
\usepackage{caption}

\usepackage{indentfirst}
\usepackage{color}
\usepackage{graphicx}
\usepackage{newlfont}

\usepackage{amssymb}
\usepackage{amsmath}
\usepackage{latexsym}
\usepackage{amsthm}
\usepackage{mathrsfs}
\usepackage{amscd}

\begin{document}
\makeatletter
\@namedef{subjclassname@2020}{%
	\textup{2020} Mathematics Subject Classification}
\makeatother

\title{Fano 3-folds and double covers by half elephants}

\author{Livia Campo}

\address{School of Mathematics \\
	KIAS \\
	85 Hoegiro, Dongdaemun-gu \\
	Seoul, 02455 \\
	Republic of Korea}

\email{liviacampo@kias.re.kr}

\keywords{Fano 3-fold, Fano index, Equivariant unprojection}

\thanks{The author would like to thank Gavin Brown, Tiago Guerreiro, Stavros Argyrios Papadakis, Kaori Suzuki, and Miles Reid for conversations and comments during the development of this work. The author was supported by EPSRC Doctoral Training Partnership, by EPSRC grant EP/N022513/ held by Alexander Kasprzyk, and by the Korea Institute for Advanced Study (KIAS), grant No. MG087901.}

\begin{abstract} 
	We construct a deformation family for each of the 34 Hilbert series of Fano 3-folds in codimension 4 having Fano index 2. In 18 cases we construct two different families, distinguished by the topology of their general members. 
\end{abstract}

\maketitle

\section{Introduction}

We work over the field of complex numbers $\C$. A $\Q$-\textit{Fano 3-fold} $X$ is a normal projective 3-dimensional variety with ample anticanonical divisor $-K_X$ and $\Q$-factorial terminal singularities.
The \textit{Fano index} of $X$ is
\begin{equation*}
\iota_X \coloneqq \max \{ q \in \Z_{\geq 1} \; : \; -K_X \sim q A \text{ for some } A \in \Cl(X) \} \; .
\end{equation*}
The landscape of Fano 3-folds having Fano index 2 is partially understood: \cite{BrownSuzuki} gives a list of possible Hilbert series of these varieties, although that paper does not confirm the existence of any particular one. We call this list $\mathscr{H}^2_{\text{BS}}$.

\subsection{Main Theorem}

We consider Fano 3-folds that are embedded primitively into weighted projective spaces in codimension 4. 
The aim of this paper is to systematically build at least one deformation family for each 34 of the Hilbert series in $\mathscr{H}^2_{\text{BS}}$. The strategy is to retrieve them from just as many deformation families of codimension 4 Fano 3-folds $X$ of Fano index 1 by performing a quotient by a $\Z/2\Z$ action $\gamma$ on the ambient spaces of these $X$'s. 

The construction we achieve is summarised in this diagram:
\begin{equation} \label{diagram index 2}
\vcenter{\xymatrix{
	& \text{codim 4} & & \text{codim 3} \\
	\text{index 1} & X \ar[d]_{\Z/2\Z}^\gamma & & Z \ar@{-->}[ll]_{\text{unprojection}} \ar[d]^{\Z/2\Z}_\gamma \\
	\text{index 2} & \tilde{X} & & \tilde{Z}
}}
\end{equation}

The Main Theorem we prove is

\begin{theorem}\label{Main Theorem}
	For each of the 34 power series $\mathscr{P} \in \mathscr{H}^2_{\text{BS}}$ there exists at least one deformation family of codimension 4 Fano 3-folds $\tilde{X}$ having Fano index 2 such that $\mathscr{P}_{\tilde{X},A} = \mathscr{P}$ for $A \in \left| -\frac{1}{2} K_{\tilde{X}} \right| $. 
	For 18 of the Hilbert series in $\mathscr{H}^2_{\text{BS}}$ there are at least two distinguished deformation families.
\end{theorem}
In practice, we prove that for each $\mathscr{P} \in \mathscr{H}^2_{\text{BS}}$ there exists at least one deformation family of Fano index 1 Fano 3-folds $X \subset w\Proj^7$ with a $\Z/2\Z$ action $\gamma$ such that $\tilde{X} \coloneqq X\!/\!\gamma$ is a deformation family of quasi-smooth Fano 3-folds having Fano index 2, and the Hilbert series of $(\tilde{X},A)$ matches $\mathscr{P}$. 
We produce $X$ by Type I \cite{PapadakisComplexes,PapadakisReidKM} and Type II \cite{PapadakisTypeII} unprojections from particular Fano 3-folds $Z$ in codimension 3 and from Fano hypersurfaces respectively. 
The key is to have~$X$ invariant under~$\gamma$. We discuss how and when this is possible in Sections \ref{Tom families}, \ref{Jerry families}, \ref{Type II}. The details are summarised in Table \ref{Table TJ index 2}.

\subsection{Framework}

Some constructions of index 2 Fano 3-folds already exist in the literature. In \cite{ProkhorovReid} Prokhorov and Reid 
have constructed one example of an index 2 Fano 3-fold in codimension 3 and one in codimension~4 performing divisorial extractions of curves in $\Proj^3$ and in a quadric in $\Proj^4$ respectively. Their argument is concluded by running the Sarkisov Program initiated by such extractions. Their construction was further generalised by Ducat in \cite{ducat2018alaprokhreid}, who obtained one new family in codimension 4, and one each in codimension~5 and~6.

There are 35 Hilbert series in $\mathscr{H}^2_{\text{BS}}$ associated to codimension 4 Fano 3-folds with Fano index~2 as in the Graded Ring Database (GRDB) \cite{grdb, AltinokBrownReidFanoK3}, of which each realise one using their constructions.
A further case\footnotemark is the Hilbert series of the smooth Fano 3-fold with GRDB ID \#41028, constructed by Iskovskih in \cite[Case 13, Table (6.5)]{Iskovskih1,Iskovskih2}. We added it to Table \ref{Table TJ index 2} for completeness. 
\footnotetext{The Hilbert series with ID \#40367 and \#40378 also appear in the GRDB in index 2 and codimension 4. However, the Fano 3-folds associated to these two Hilbert series cannot embed in codimension 4 in the weighted projective space suggested by the Graded Ring Database. Indeed, neither our method nor the method in \cite{CoughlanDucat} construct them. These two Fano 3-folds may exist embedded in higher codimension.} 

Of the 34 Hilbert series relative to codimension 4 index 2 Fano 3-folds constructed with our method, 32 admit a double cover $X$ obtained via Type I unprojection from $Z$ in codimension 3. There are between two and four possible distinguished deformation families of $X$, associated to just as many formats (Tom and Jerry) of the antisymmetric matrix $M$ defining the equations of $Z$ \cite{T&Jpart1}.
Our construction is applicable to one or more of these deformation families: the number changes case by case, and the full extent of the result is outlined in Table \ref{Table TJ index 2}. Accordingly, $\tilde{X}$ has one or more distinguished deformation families. We give criteria to find deformation families for $\tilde{X}$ in Sections \ref{Tom families} and \ref{Jerry families}. In a similar fashion to the phenomena occurring in \cite{T&Jpart1}, some formats do not give rise to deformation families with the desirable features (terminality, for instance); we refer to this as to \textit{failure}. Fundamentally, most information is included in the geometry of $Z$.

On the other hand, two codimension 4 index 2 Fano 3-folds in the GRDB have a double cover $X$ obtained via Type II$_2$ unprojection. Performing the unprojection in this case is more complicated, and we refer to \cite{PapadakisGeneralTheoryUnproj,PapadakisTypeII,Taylor}. The double cover construction still works in this case: this is the content of Section \ref{Type II}.

The Hilbert series with GRDB ID \#40933, \#40663 are in the  overlap with the results of \cite{ProkhorovReid} and \cite{ducat2018alaprokhreid}. In particular, Ducat finds two different deformation families of \#40663, and we retrieve them here with our method (cf \cite[Section 3]{ducat2018alaprokhreid}).

A different approach to the construction of Fano 3-folds is given by Coughlan and Ducat in \cite{CoughlanDucat}, where the authors employ rank 2 cluster algebras to find families of Fano 3-folds in codimension 4 and 5, also in index 2. 
In addition, in certain cases they are able to determine different deformation families associated to the same Hilbert series (what they call \textit{cluster formats}). Cluster formats mimic the Tom and Jerry formats of \cite{T&Jpart1}. 
Our work finds more deformation families that are not realised by cluster formats, which could possibly have higher Picard rank.

\subsection{Details of the construction} \label{Tables}

The following table summarises all possible Tom and Jerry types for index 2 Fano 3-folds in codimension 4. In the last column it also records which index~1 formats (mostly Jerry) of \cite{TJBigTable} do not give rise to families in index~2. In the case where the same 
index 2 Fano 3-fold admits more than one format of the same type, we specify the centre for each format. 
The column \#$_\text{I}$ denotes the total number of distinguished deformation families constructed by Theorems \ref{criterion for Tom} and \ref{criterion for Jerry}. 
We write "n/a" in the cases in which the Tom and Jerry construction is not applicable. In particular, \#$_\text{II}$ represents the total number of distinguished deformation families when there is no Type I projection. We construct one in Section \ref{Type II}, but there could possibly be others. 
For the $\bullet_{ij}$ notation see end of Section \ref{Background} below, and \cite{TJBigTable}.

\begin{adjustwidth}{-30em}{-4em}
	\begin{center}
		\begin{longtable}{| c | c || c || c | c || c |}
			\caption{Deformation families of index 2 Fano 3-folds in codimension 4 \\and corresponding index 1 double covers} \label{Table TJ index 2}
			
			\centering
			\endfirsthead
			\endhead
			\hline
			Index 2 & Index 1 & \#$_\text{I}$ & Tom & Jerry & Failures \\
			\hline \hline
			39557 & 327 & 2 & $T_3$ & $J_{24}$ & none \\
			\hline
			39569 & 512 & \multicolumn{4}{c|}{n/a: no Type I projection. \#$_\text{II} \geq 1$} \\
			\hline
			39576 & 569 & 1 & $T_1$ & none & $J_{25}$ \\
			\hline
			39578 & 574 & 2 & $T_1$ & $J_{24} \bullet_{12}$ & $J_{45}$ \\
			\hline
			39605 & 869 & 2 & $T_4$  & $J_{13}$ & none \\
			\hline
			39607 & 872 &  \multicolumn{4}{c|}{n/a: no Type I projection. \#$_\text{II} \geq 1$} \\
			\hline
			39660 & 1158 & 2 & $T_5$ & $J_{12}$ & none \\
			\hline
			39675 & 1395 & 2 & $T_5$ & none & $J_{12}$ \\
			\hline
			\multirow{2}{2.5em}{39676} & \multirow{2}{2.0em}{1401} & \multirow{2}{0.5em}{1} & $\frac{1}{5} \colon T_2$ & \multirow{2}{2.5em}{none} & $\frac{1}{5} \colon J_{24} \bullet_{12}$, $J_{45}$  \\
			& & & $\frac{1}{7} \colon T_4$ & &  $\frac{1}{7} \colon J_{12} \bullet_{15}$, $J_{24}$  \\
			\hline
			39678 & 1405 & 1 & $T_1$ & none & $J_{24}$ \\
			\hline
			39890 & 4810 & 2 & $T_3$ & $J_{24}$ & $J_{14} \bullet_{13}$ \\
			\hline
			39898 & 4896 & 2 & $T_3$ & $J_{24}$ & $J_{14} \bullet_{13}$ \\
			\hline
			\multirow{2}{2.5em}{39906} & \multirow{2}{2.0em}{4925} & \multirow{2}{0.5em}{1} & $\frac{1}{7}(1,1,6) \colon T_2$ & \multirow{2}{2.5em}{none} & $\frac{1}{7}(1,1,6) \colon J_{34}$ \\
			& & & $\frac{1}{7}(1,3,4) \colon T_1$ & & $\frac{1}{7}(1,3,4) \colon J_{35}$ \\
			\hline
			39912 & 4938 & 1 & $T_2$ & none & $J_{12}$ \\
			\hline
			\multirow{2}{2.5em}{39913} & \multirow{2}{2.0em}{4939} & \multirow{2}{0.5em}{2} & $\frac{1}{5} \colon T_1$ & $\frac{1}{5} \colon J_{25} \bullet_{24}$ & $\frac{1}{5} \colon J_{35}$ \\
			& & & $\frac{1}{7} \colon T_2$ & $\frac{1}{7} \colon J_{14} \bullet_{13}$ & $\frac{1}{7} \colon J_{24}$ \\
			\hline
			39928 & 4987 & 1 & $T_5$ & none & $J_{12}$ \\
			\hline
			\multirow{2}{2.5em}{39929} & \multirow{2}{2.0em}{5000} & \multirow{2}{0.5em}{1} & $\frac{1}{5} \colon T_2$ & \multirow{2}{2.5em}{none} & $\frac{1}{5} \colon J_{45}$ \\
			& & & $\frac{1}{9} \colon T_4$ & & $\frac{1}{9} \colon J_{24}$ \\
			\hline
			39934 & 5052 & 2 & $T_1$ & $J_{23} \bullet_{13}$ & $J_{24}$ \\
			\hline
			\multirow{2}{2.5em}{39961} & \multirow{2}{2.0em}{5176} & \multirow{2}{0.5em}{1} & $\frac{1}{5} \colon T_2$ & \multirow{2}{2.5em}{none} & $\frac{1}{5} \colon J_{35}$ \\
			& & & $\frac{1}{7} \colon T_3$ & & $\frac{1}{7} \colon J_{25}$ \\
			\hline
			39968 & 5260 & 2 & $T_5$ & $J_{13}$ & none \\
			\hline
			\multirow{2}{2.5em}{39969} & \multirow{2}{2.0em}{5266} & \multirow{2}{0.5em}{2} & $\frac{1}{5} \colon T_3$ & $\frac{1}{5} \colon J_{24} \bullet_{25}$ & $\frac{1}{5} \colon J_{45}$ \\
			& & & $\frac{1}{7} \colon T_4$ & $\frac{1}{7} \colon J_{13} \bullet_{15}$ & $\frac{1}{7} \colon J_{34}$ \\
			\hline
			\multirow{3}{2.5em}{39970} & \multirow{3}{2.0em}{5279} & \multirow{3}{0.5em}{1} & $\frac{1}{3} \colon T_1$ & \multirow{3}{2.5em}{none} & $\frac{1}{3} \colon J_{45}$ \\
			& & & $\frac{1}{5}(1,1,4) \colon T_2$ & & $\frac{1}{5}(1,1,4) \colon J_{34}$ \\
			& & & $\frac{1}{5}(1,2,3) \colon T_1$ & & $\frac{1}{5}(1,2,3) \colon J_{35}$ \\
			\hline
			\multirow{2}{2.5em}{39991} & \multirow{2}{2.0em}{5516} & \multirow{2}{0.5em}{1} & $\frac{1}{3} \colon T_1$ & \multirow{2}{2.5em}{none} & $\frac{1}{3} \colon J_{45}$ \\
			& & & $\frac{1}{7} \colon T_3$ & & $\frac{1}{7} \colon J_{13}$ \\
			\hline
			\multirow{2}{2.5em}{39993} & \multirow{2}{2.0em}{5519} & \multirow{2}{0.5em}{2} & $\frac{1}{3} \colon T_1$ & $\frac{1}{3} \colon J_{34}$ & $\frac{1}{3} \colon J_{35}$ \\
			& & & $\frac{1}{5} \colon T_2$ & $\frac{1}{5} \colon J_{12}$ & $\frac{1}{5} \colon J_{13}$ \\
			\hline
			40360 & 10963 & 2 & $T_3$ & $J_{24}$ & $J_{14} \bullet_{13}$ \\
			\hline
			40370 & 11004 & 1 & $T_2$ & none & $J_{12}$ \\
			\hline
			\multirow{2}{2.5em}{40371} & \multirow{2}{2.5em}{11005} & \multirow{2}{0.5em}{2} & $\frac{1}{3} \colon T_1$ & $\frac{1}{3} \colon J_{25} \bullet_{24}$ & $\frac{1}{3} \colon J_{35}$ \\
			& & & $\frac{1}{5} \colon T_2$ & $\frac{1}{5} \colon J_{14} \bullet_{13}$ & $\frac{1}{5} \colon J_{24}$ \\
			\hline
			40399 & 11104 & 1 & $T_5$ & none & $J_{24}$ \\
			\hline
			\multirow{2}{2.5em}{40400} & \multirow{2}{2.5em}{11123} & \multirow{2}{0.5em}{1} & $\frac{1}{3} \colon T_3$ & \multirow{2}{2em}{none} & $\frac{1}{3} \colon J_{24} \bullet_{25}$, $J_{45}$ \\
			& & & $\frac{1}{5} \colon T_4$ & & $\frac{1}{5} \colon J_{12} \bullet_{15}$, $J_{24}$ \\
			\hline
			40407 & 11222 & 2 & $T_1$ & $J_{23} \bullet_{13}$ & $J_{24}$ \\
			\hline
			40663 & 16206 & 2 & $T_4$ & $J_{23}$ & $J_{12} \bullet_{15}$ \\
			\hline 
			40671 & 16227 & 1 & $T_2$ & none & $J_{12}$ \\
			\hline
			40672 & 16246 & 2 & $T_2$ & $J_{15} \bullet_{14}$ & $J_{25}$ \\
			\hline
			40933 & 24078 & 2 & $T_5$ & $J_{12}$ & $T_1$ \\
			\hline
			41028 & n/a & \multicolumn{4}{c|}{n/a} \\
			\hline
			
		\end{longtable}
	\end{center}
\end{adjustwidth}

\section{Background} \label{Background}

The lack of structure theorems for Fano 3-folds in codimension greater than 3 has forced the search for new approaches to produce their equations explicitly. Unprojections are a technique to retrieve equations for Fano 3-folds in codimension 4 from Fano 3-folds in lower codimension. 
They were firstly studied by Kustin and Miller \cite{KustinMiller}, and later on by Reid and Papadakis \cite{PapadakisReidKM,PapadakisComplexes}. 
There are different kinds of unprojections: the most widely employed are called Type I unprojections. 

Type I unprojections are initiated by the following type of data.
\begin{itemize}
	\item A fixed projective plane $D \coloneqq \Proj^2(a,b,c) \subset \Proj^6(a,b,c, d_1, \dots, d_4)$ with coordinates $x_1, x_2, x_3,$ $y_1, \dots,y_4$ respectively, and defined by the ideal $I_D \coloneqq \langle y_1,y_2,y_3,y_4 \rangle$.
	\item A family $\mathcal{Z}$ of codimension 3 Fano 3-folds $Z \subset w\Proj^6$, each defined by the five maximal pfaffians of a skew-symmetric $5 \times 5$ syzygy matrix $M$ whose entries $(a_{ij})$ have weights 
	\begin{equation*}
	\left(
	\begin{array}{c c c c}
	m_{12} & m_{13} & m_{14} & m_{15} \\
	& m_{23} & m_{24} & m_{25} \\
	& & m_{34} & m_{35} \\
	& & & m_{45}
	\end{array}
	\right) \; .
	\end{equation*}
\end{itemize}

In this context, two kinds of formats arise for $M$, based on conditions on its entries. These are the so-called \textit{Tom and Jerry formats}. If $M$ is in one of these formats, then $D \subset Z \subset w\Proj^6$

\begin{definition}[\cite{T&Jpart1}, Definition 2.2]\label{TJ definition}
	A $5 \times 5$ skew-symmetric matrix $M$ is in Tom$_k$ format if and only if each entry $a_{ij}$ for $i,j \neq k$ is in the ideal $I_D$.
	
	It is in Jerry$_{kl}$ format if and only if $a_{ij} \in I_D$ for either $i$ or $j$ equals $k$ or $l$. 
	If $M$ is in Jerry$_{kl}$ format, we call \textit{pivot entry} the entry $a_{kl} \in I_D$. 
\end{definition}

Recall the definition of Fano 3-fold of \textit{Tom type} (respectively, of \textit{Jerry type}).

\begin{definition}[\cite{CampoSarkisov}, Definition 2.2] \label{Tom type def}
	Let $X$ be a codimension 4 index 1 Fano 3-fold $X$ listed in the table \cite{TJBigTable}. We say $X$ is \textit{of Tom (Jerry) type} if it is obtained as Type I unprojection of the codimension 3 pair $Z \supset D$ in a Tom (Jerry) family \cite{T&Jpart1,PapadakisReidKM}. 
	The image of $D \subset Z$ in $X$ is called \textit{Tom (Jerry) centre}: it is a cyclic quotient singularity $p \in X$. In the unprojection setup $D \subset Z$, $D$ is a complete intersection of four linear forms of weight $d_1, \dots, d_4$. Such $X$ of Tom (Jerry) type is said to be \textit{general} if $Z \supset D$ is general in its Tom (Jerry) family.
\end{definition}
Note that $X$ is quasi-smooth \cite[Theorem 3.2]{T&Jpart1} and Gorenstein \cite[Theorems 5.6 and 5.14]{PapadakisComplexes}; our arguments begin within this framework.

The formats denoted by $\bullet_{ij}$ in Table \ref{Table TJ index 2} are those where: either, the matrix $M$ can be manipulated by row/column operations so that the entry $m_{ij}$ is 0; or, there is no polynomial that fits entry $m_{ij}$ such that it satisfies the Tom and Jerry constraints.

\section{Double covers} \label{double covers}

Let $X$ be a $\Q$-Fano 3-fold in codimension 4 anticanonically embedded in a weighted $w\PP^7$ having Fano index 1 and such that $h^0(X,-K_X) \geq 1$. Thus we assume that $w\PP^7 = \Proj^7(1,b,c, d_1, \dots, d_4, r)$ with homogeneous coordinates $x_1, x_2, x_3, y_1, \dots, y_4, s$. 
Consider the $\Z/2\Z$ action on $w\Proj^7$ that changes the sign of $x_1$ of weight $\wt(x_1)=1$. 
\begin{equation} \label{Z2 action}
\gamma \colon \left( x_1, x_2, x_3, y_1, y_2, y_3, y_4, s \right) \longmapsto \left( -x_1, x_2, x_3, y_1, y_2, y_3, y_4, s \right) \; .
\end{equation}
First we  describe the quotient of the ambient space of $X$ by the $\Z/2\Z$ action $\gamma$.

\begin{lemma}
	The $\Z/2\Z$ quotient of $\Proj^7(1,b,c, d_1, \dots, d_4, r)$ via $\gamma$ is the weighted projective space $\Proj^7(2,b,c, d_1, \dots, d_4, r)$ with coordinates $\xi, x_2,x_3,y_1,\dots,y_4,s$ respectively, where $\xi:= x_1^2$.
\end{lemma}
\begin{proof}
	Consider the affine patches of $\Proj^7(1,b,c, d_1, \dots, d_4, r) = \proj \C(x_1,x_2,x_3, y_1, \dots, y_4,s)$. For instance,
	\begin{equation*}
	\mathcal{U}_{x_2}\coloneqq \{ x_2 \not= 0 \} = \Spec \C \left[ x_1, \hat{x}_2,x_3,y_1,\dots,y_4,s \right]^{\mu_b} \cong \A^7/\mu_b
	\end{equation*}
	where $\mu_b$ is the finite cyclic group of order $b$. The affine patches relative to the coordinates $x_3,y_1,\dots,y_4,s$ are analogous. 
	They are invariant under $\gamma$ if and only if $x_1$ appears with even powers.
	Such affine patches in which powers of $\xi$ appear are exactly the affine patches of $\Proj^7(2,b,c, d_1, \dots, d_4, r)$, with the new coordinate $\xi$. So, 
	\begin{align*}
	\Proj^7(1,b,c, d_1, \dots, d_4, r)/\gamma &= \proj \C(x_1,x_2,x_3, y_1, \dots, y_4,s)^\gamma = \proj \C(x_1^2,x_2,x_3, y_1, \dots, y_4,s) \\
	&= \Proj^7(2,b,c, d_1, \dots, d_4, r) \; .
	\end{align*}
	We identify the homogeneous coordinates of $X$ and $\tilde{X}$ and define $\xi:= x_1^2$.
\end{proof}

Hence, for $X \subset \Proj^7(1,b,c, d_1, \dots, d_4, r)$ to be invariant under $\gamma$, the variable $x_1$ must appear only with even powers. When this is the case, then the quotient $\tilde{X} \coloneqq X/\gamma$ sits inside the weighted projective space $\Proj^7(2,b,c, d_1, \dots, d_4, r)$ and is a Fano 3-fold: its anticanonical divisor $-K_{\tilde{X}}$ is a multiple of an ample divisor (see proof of Lemma \ref{index of Xtilde}), it has terminal singularities (see Lemma \ref{Xtilde has terminal sings}), it is $\Q$-factorial.

Let $\Proj_{\text{even}}$ be the weighted projective space defined by the vanishing of all the coordinates of $w\Proj^7$ with odd weight, except for $x_1$. 
In this section, let us assume that it is possible to realise $X \subset \Proj^7(1,b,c, d_1, \dots, d_4, r)$ as invariant under $\gamma$. To fix ideas, we call the defining equations of $X = \{ f_i=0 \}_{i=1}^9$ and the equations of $\tilde{X} = \{ \tilde{f}_i=0 \}_{i=1}^9$, where $\tilde{f}_i = f_i(\xi, \dots, s)$. 
We can therefore draw the following conclusions.

\begin{lemma} \label{index of Xtilde}
	The 3-fold $\tilde{X}$ has Fano index 2, and $-K_{\tilde{X}}$ is ample.
\end{lemma}
\begin{proof}
	Consider the quotient map $\varphi \colon X \rightarrow \tilde{X}$. Then, the anticanonical divisor of $X$ is given by $-K_X = -\varphi^*K_{\tilde{X}} - R$ where $R$ is the ramification divisor. In our case, $-K_X = \{x_1=0 \} \sim \mathcal{O}(1)$, and the ramification divisor is $R= \{x_1=0 \}$. Therefore, $-\varphi^*K_{\tilde{X}} = 2 \{x_1=0 \}$. This implies that $-K_{\tilde{X}} = \{\xi=0 \}\sim \mathcal{O}(2)$: thus, $\tilde{X}$ has index 2. 
	Moreover, since $-K_X$ is ample, so is $-K_{\tilde{X}}$.
\end{proof}

\begin{lemma} \label{Xtilde is quasismooth}
	If $X$ is quasismooth, then $\tilde{X}$ is quasismooth.
\end{lemma}
\begin{proof}
	For $p \in w\Proj^7$ consider the Jacobian matrix $J_X$ of $X$ at the point $p$ as in \cite[Section 5, pp. 31-32]{Hartshorne}. 	Define the variety $V_X$ as
	\begin{equation*}
	V_X := \{ p \in w\Proj^7 \; : \; \rk \left( J_X |_p \right) < \codim (X) \} \; .
	\end{equation*}
	The affine cone of $V_X$ is the singular locus of the affine cone of $X$. 
	By definition, if $V_X$ is empty, then $X$ is quasismooth. 
	Suppose $V_X$ empty. 
	For each equation $f_i$ of $X$, $\frac{\partial f_i}{\partial x_1} = \frac{\partial \tilde{f}_i}{\partial \xi}\frac{\partial \xi}{\partial x_1}$, and $\frac{\partial \xi}{\partial x_1}= 2 x_1$. 
	The difference between the Jacobian matrices $J_X$ and $J_{\tilde{X}}$ lies in the column relative to the derivative by $x_1$. Suppose $x_1 \not= 0$; then, the rank of $J_X$ is equal to the rank of $J_{\tilde{X}}$. 
	If instead $x_1 = 0$, certain entries of the $\frac{\partial}{\partial x_1}$ column of $J_X$ might vanish for $x_1 = 0$, whereas they would be just constant in $J_{\tilde{X}}$.
	Thus, for $x_1 = 0$ we have that $\rk J_X \leq \rk J_{\tilde{X}}$; so, $\tilde{X}$ is quasismooth if $X$ is.
\end{proof}

\begin{lemma} \label{Xtilde has terminal sings}
	The 3-fold $\tilde{X}$ has terminal singularities.
\end{lemma}
\begin{proof}
	The fixed locus of the group action $\gamma$ is $\Fix(\gamma) = \{x_1=0 \} \cup \Proj_{\text{even}}$. 
	We want to study the intersection $X \cap \Fix(\gamma)$.
	
	From \cite[Theorems 5.6 and 5.14]{PapadakisComplexes} we have that $X$ is Gorenstein; thus, the base locus $\Bs|-K_X|$ does not contain any non-Gorenstein points. Then, all the cyclic quotient singularities of $X$ lie at coordinate points different from $P_{x_1}$. Thus, they all lie inside the locus $\{x_1=0 \}$. 
	
	On the other hand, $ X \cap \Proj_{\text{even}} = \emptyset$. Suppose that $ X \cap \Proj_{\text{even}} \not= \emptyset$. If $\dim X \cap \Proj_{\text{even}} \geq 1$ we reach a contradiction, because $X$ is terminal. 
	If instead $\dim X \cap \Proj_{\text{even}} = 0$, that is, $ X \cap \Proj_{\text{even}}$ is a finite number of points, such points would be singularities of type $\frac{1}{r} (1,a,b)$ with even $r$, in accordance to the basket of singularities of $X$. However, none of the codimension 4 index 1 candidate double covers $X$ listed in the second row of Table \ref{Table TJ index 2} has a basket that contains cyclic quotient singularities of even order. Therefore, we conclude that $ X \cap \Proj_{\text{even}} = \emptyset$. 
	Thus, $\tilde{X}$ has terminal singularities.
\end{proof}

The above lemmas prove the following Theorem, which will be crucial in the rest of this paper to prove a criterion for finding the deformation families of $\tilde{X}$.
\begin{theorem} \label{Xtilde in grdb}
	If $X \subset \Proj^7(1,b,c, d_1, \dots, d_4, r)$ has an invariant realisation under $\gamma$, then $\tilde{X}$ is a terminal $\Q$-factorial Fano 3-fold. Hence, it is in the Graded Ring Database.
\end{theorem}

In contrast, our method does not produce the index 2 Fano 3-folds in codimension 3 in the Graded Ring Database. Restrict the action $\gamma$ in \ref{Z2 action} to the ambient space $w\Proj^6$ of $Z$. With a little abuse of notation, we still call the restriction $\gamma$.

\begin{proposition} \label{Ztilda not terminal}
	The index 2 Fano 3-fold $\tilde{Z}:= Z\!/\!\gamma$ in codimension 3 is not terminal.
\end{proposition}
\begin{proof}
	Consider the fixed locus $\Fix(\gamma) = \{x_1=0 \} \cup \Proj_{\text{even}}$ of $\gamma$, and its intersection with $Z$. The plane $D \cong \PP(1,b,c)$ is contracted via unprojection to the terminal singularity of type $\frac{1}{r}(1,b,c)$ in $X$, where $c=r-b$ and $(b,r)=1$. Thus, either $b$ or $c$ is even. 
	In the quotient we have that $\tilde{D}:= D\!/\!\gamma \cong \PP(2,b,c)$. 
	Therefore, $\tilde{Z}$ and $\Proj_{\text{even}}$ meet on $D$ along a line constituted by $\frac{1}{2}$ singularities.
\end{proof}

The phenomenon described in Proposition \ref{Ztilda not terminal} was already anticipated in \cite[Sections 6.4, 6.5]{ProkhorovReid}. 
There are only two Hilbert series corresponding to terminal index 2 codimension 3 Fano 3-folds: one is smooth, and is constructed in \cite{Iskovskih1} and \cite{TakaoFujita}. The other one was constructed by Ducat in \cite{ducat2018alaprokhreid}.

We would like to stress that $X$ is general in its Tom or Jerry formats, as in Sections \ref{Tom families}, \ref{Jerry families} (cf \cite[Section 3]{CampoSarkisov}, \cite[Section 4]{T&Jpart1}). Thus, provided the additional condition of invariance of $X$ under $\gamma$, $\tilde{X}$ is general.

\section{Tom families} \label{Tom families}

Let $X$ be a codimension 4 $\Q$-Fano 3-fold having at least one Type I centre, and suppose that $X$ is obtained by the corresponding Type I unprojection from a divisor $D$ inside a codimension 3 $\Q$-Fano 3-fold $Z$ in Tom format.

In this section we show a criterion to determine which Tom formats of $Z$ induce a double cover $X$ for the corresponding $\tilde{X}$ of Fano index 2. This produces deformation families for $\tilde{X}$. To this purpose, it is crucial to understand the geometry of $Z$, and of the nodes lying on the divisor $D \subset Z$. 
The following statement holds for both Tom and Jerry formats, and will also be used in Section \ref{Jerry families} below.  

\begin{lemma} \label{nodes not fixed}
	If a general codimension 3 Fano 3-fold $Z$ in either Tom or Jerry format is invariant under the $\Z/2\Z$ action $\gamma$, then the nodes on the divisor $D \subset Z$ are not fixed by $\gamma$. 
	In particular, they are pairwise-identified in the quotient $\tilde{Z}$, and $\#\{\text{nodes on } Z\}=2 \cdot\#\{\text{nodes on } \tilde{Z}\}$.
\end{lemma}
\begin{proof}
	From \cite[Theorem 3.2, Lemma 3.1]{T&Jpart1}, the nodes of $Z$ only lie on the divisor $D$. 	
	They are given by the $3 \times 3$ minors of the Jacobian matrix $J\big|_Z$ restricted to $D$, i.e. $\bigwedge^3 J |_D = \underline{0}$. Their equations are not $\gamma$-invariant, because the entries of the $\frac{\partial}{\partial x_1}$ column of $J_Z$ are not $\gamma$-invariant. Thus, the nodes are not fixed by the action. 
	In addition, for $\tilde{D} \coloneqq D/\gamma$, the nodes on $\tilde{D} \subset \tilde{Z}$ are given by $\bigwedge^3 \tilde{J} |_{\tilde{D}} = \underline{0}$. Such equations depend on $\xi$. As a consequence, the nodes on $D \subset Z$ are pairwise identified by $\gamma$ in the quotient and the number of nodes on $Z$ is twice the number of nodes on $\tilde{Z}$.
\end{proof}

There are either one or two Tom deformation families for each $X$ in index 1 \cite{T&Jpart1,TJBigTable}.
For each of the 32 index 1 Fano 3-folds $X$ in the GRDB that are candidates to be double covers of just as many $\tilde{X}$ in index 2, there is exactly one Tom deformation family that realises the double cover. 
The Tom families coming from codimension 3 Tom formats having odd number of nodes are automatically excluded by Lemma \ref{nodes not fixed}. 

Conversely, it is also possible to prove the following lemma.
\begin{lemma} \label{from Z to Z2Z invariance}
	Let $Z$ be a codimension 3 terminal Fano 3-fold in Tom format having Fano index 1 such that the number of nodes on $D \subset Z$ is even. If a Type I unprojection $X$ of $Z$ has basket of singularities formed only by cyclic quotient singularities with odd order, then it is possible to realise $Z$ as $\gamma$-invariant in such a way that $\tilde{X}$ is terminal.
\end{lemma}
\begin{proof}
	In \cite{T&Jpart1} the authors provide a \verb|Magma| code (available at \cite{grdb} $\rightarrow$ Downloads) that automatises the check for failure of Tom and Jerry formats in the sense of \cite[Section 5]{T&Jpart1}. The code essentially follows the steps of the proof of \cite[Theorem 3.2]{T&Jpart1}, as illustrated in \cite[Section 8]{T&Jpart1}. 
	We run the code on the codimension 3 Fano 3-folds $Z$ from which we obtain candidate double covers $X$ for $\tilde{X}$, imposing that $x_1$ appears only with even powers, i.e. replacing $x_1$ with $\xi$.
	
	This shows that only the Tom formats having even number of nodes on $D \subset Z$ descend to Fano index 2 in codimension 4. Moreover, it is needed that $Z$ is in Tom format simply because these hypotheses are not enough when considering Jerry formats. Indeed, there are Jerry formats that satisfy the conditions on basket and nodes of Lemma \ref{from Z to Z2Z invariance} that do not descend to Fano index 2 (see for instance \#11104 Jerry$_{12}$ and the other families mentioned in Subsection \ref{Jerry failure}). 
	Once a successful format is found, any general choice of entries of the matrix $M$ would still work (cf \cite[Section 8]{T&Jpart1}). 
	
	The hypotheses on the basket of singularities of $Z$ ensures that $\tilde{X}$ is terminal (cf proof of Lemma \ref{Xtilde has terminal sings} and \cite[Lemma 1.2 (3)]{SuzukiIndexBound}).
\end{proof}
Recall from the proof of Lemma \ref{Xtilde has terminal sings} that $\Fix(\gamma) = \{x_1=0 \} \cup \Proj_{\text{even}}$, and that $X \cap \{x_1=0 \}$ contains all the cyclic quotient singularities of $X$. Thus, the basket of $X$ is fixed by $\gamma$. Therefore, it is important to observe that the following situations do not occur (cf \cite[Lemma 1.2 (3)]{SuzukiIndexBound}): a singularity of even order $2n$ on $X$ descends to a singularity of order $4n$ on $\tilde{X}$; the singularities of even order on $X$ are pairwise identified on $\tilde{X}$. 
This, together with Lemmas \ref{nodes not fixed}, \ref{from Z to Z2Z invariance}, therefore proves the following theorem.

\begin{theorem} \label{criterion for Tom}
	A general codimension 3 terminal Fano 3-fold $Z$ of index 1 in Tom format can be realised as invariant under $\gamma$ if and only if the number of nodes on $D \subset Z$ is even and the basket of its Type I unprojection $X$ exclusively contains cyclic quotient singularities of odd order.
\end{theorem}

As summarised in Table \ref{Table TJ index 2}, each of the 32 codimension 4 index 1 Fano 3-folds that are candidates to be double covers, have one and only one Tom deformation family as described in Theorem \ref{criterion for Tom}. Except for the Hilbert series with GRDB ID \#24078, the only Tom deformation family that admits the quotient by $\gamma$ is the first Tom family as in \cite{TJBigTable}. Instead, the Fano \#24078 only admits its second Tom family Tom$_5$; the deformation family Tom$_5$ of \#24078 has Picard rank 2 by \cite{brownP2xP2}.

\section{Jerry families} \label{Jerry families}

In this section we examine codimension 4 terminal Fano 3-folds $X$ of Jerry type, and we determine which Jerry families give rise to a deformation family for $\tilde{X}$ in index 2.

Suppose that $Z$ is a terminal codimension 3 Fano 3-fold in Jerry format having at least one Type I centre, and call $P$ the weight of the pivot entry (see Definition \ref{TJ definition}). The criterion to determine which Jerry deformation families of $X$ descend to $\tilde{X}$ depends on whether the following condition is satisfied or not.

\begin{condition} \label{condition c}
	There exists a generator $y_k$ of the ideal $I_D$ whose weight is equal to $P$.
\end{condition}
What has been discussed for Tom families still holds here, but some care is needed with regards to some features of the Jerry formats. 
In particular, the structure of the proof of Theorem \ref{criterion for Tom} can also be replicated in the Jerry case. However, there are some Jerry families, listed in Subsection \ref{Jerry failure}, that satisfy the hypotheses of Theorem \ref{criterion for Tom} but do not descend to Fano index 2. 
We therefore have the following theorem, proved in the next subsection.

\begin{theorem} \label{criterion for Jerry}
	A general codimension 3 terminal Fano 3-fold $Z$ of index 1 in Jerry format (except for \#24077 Jerry$_{12}$) can be realised as invariant under $\gamma$ if and only if the following statements are simultaneously satisfied: the number of nodes on $D \subset Z$ is even; the basket of its Type I unprojection $X$ exclusively contains cyclic quotient singularities of odd order; Condition \ref{condition c} is fulfilled; the weight $P$ of the pivot entry is even.
\end{theorem}

\subsection{Proof of Theorem \ref{criterion for Jerry} and failure of Jerry formats} \label{Jerry failure}

The first check we do in order to exclude a Jerry format is to verify that it has even number of nodes and the basket of singularities of $Z$ exclusively contains singularities of odd order. 
Afterwards, we proceed with checking the properties of the pivot entry. The check on the nodes excludes the majority of the failing Jerry formats. Only a few are left, which present different reasons for failure. 

The presence of a pivot entry with even weight is crucial for having reduced singularities on $D$. If $P$ is odd, of the three entries of $M$ that are not in $I_D$ two have odd weight and one has even weight for homogeneity of the maximal pfaffians. Imposing that $x_1$ appears only with even powers constitutes a further constraint that is often not compatible with the Jerry format: in other words, it can happen that, for a certain entry of $M$ not in $I_D$, it is not possible to find a polynomial not in $I_D$ of the right degree; this therefore forces a 0 in that entry (see below). Furthermore, row/column operations on $M$ show that other entries can be made 0. Such configurations of entries of $M$ cause $Z$ to have non-reduced singularities (cf \cite[Section 5]{T&Jpart1}).

We say that a certain format for a codimension 3 Fano 3-fold $Z$ \textit{fails} if the deformation family of $Z$ associated to that format does not enjoy the following properties: terminal, nodal on $D$ with reduced singularities. The reasons for formats' failures is extensively described in \cite[Section 5]{T&Jpart1}. 
In this paper, we have additional reasons for failure: not having even number of nodes on $D$, not satisfying Condition \ref{condition c} (in the Jerry case), not having the pivot entry with even weight. 
Except for \#24078, all Jerry formats for which $Z$ has even number of nodes on $D$ also satisfy Condition \ref{condition c}.
Here we want to discuss the Jerry formats for which the conditions \ref{condition c} and the even nodes are satisfied, but do not have the pivot entry with even weight. 

There are exactly seven formats that fall in this case. We run an explicit check on all of them using the reasons for failure in \cite[Section 5]{T&Jpart1}. 
These failing formats all give rise to non-reduced singularities on $D$. 
We specify the Type I centre in case of ambiguity, and we refer to \cite{TJBigTable} for the grading of the matrix $M$. 
\vspace{0.5em}

\paragraph{\textbf{\#1401, $\frac{1}{7}(1,2,5)$, Jerry$_{24}$}}
The ambient space of $Z$ in Jerry$_{24}$ format is $\Proj^6(1,2,3,3,4,5,5)$ with homogeneous coordinates $x_1,x_2,y_1,y_2,y_3,y_4,x_3$ respectively, and the ideal $I_D$ is generated by $y_1,y_2,y_3,y_4$. The entry $a_{13}$ has degree 3, and we want the coordinate $x_1$ to appear only with even powers in order to perform the double cover construction in Section \ref{double covers}. Thus, there is no way to fill entry $a_{13}$ with an homogeneous polynomial of degree 3 not in $I_D$, so $a_{13}$ must be 0. 
On the other hand, the entries $a_{14}, a_{23}$ must both be equal to $y_3$, as there is no other polynomial of degree 4 in $I_D$ in the available coordinates. 
Thus, the format Jerry$_{24}$ fails for the reason in \cite[Section 5.2 (4)]{T&Jpart1}.

\paragraph{\textbf{\#1401, Jerry$_{45}$}}
For $Z$ in Jerry$_{45}$ the ambient space is $\Proj^6(1,2,3,3,4,5,7)$ with homogeneous coordinates $x_1,x_2,x_3,y_1,y_2,y_3,y_4$ respectively. The entries $a_{12}, a_{13}$ of $M$ are both equal to $x_3$ because there is no other polynomial of degree 3 not in the ideal $I_D$ that is only in $x_1,x_2,x_3$. 
In addition, the entry $a_{34}$ can be made 0 via row/column operations on $M$. Thus, Jerry$_{45}$ fails because of \cite[Section 5.2 (4)]{T&Jpart1}.

\paragraph{\textbf{\#1405, Jerry$_{24}$}}
Since there are only two generators of $I_D$ in degree 5, either $a_{25}$ or $a_{34}$ can be made 0 via row/column operations on $M$. Also $a_{14} = a_{23}$ because there is only one choice for a degree 4 polynomial in $I_D$.
Hence, Jerry$_{24}$ fails due to \cite[Section 5.2 (4)]{T&Jpart1}.

\paragraph{\textbf{\#4987, Jerry$_{12}$}}
The entry $a_{23}$ is 0 after row/column operations, and $a_{34}=0$ because there is no degree 7 polynomial not in $I_D$. 
So, Jerry$_{12}$ fails due to \cite[Section 5.2 (3)]{T&Jpart1}.

\paragraph{\textbf{\#11104, Jerry$_{12}$}}
The entry $a_{23}$ is 0 after row/column operations, and $a_{34}=0$ because there is no degree 3 polynomial not in $I_D$. 
So, Jerry$_{12}$ fails due to \cite[Section 5.2 (3)]{T&Jpart1}.

\paragraph{\textbf{\#11123, Jerry$_{12} \bullet_{15}$}}
The entry $a_{15}$ is already 0 in this format. Moreover, $a_{34} = a_{35}$ because there is only one choice for a degree 3 polynomial in $I_D$. Thus, Jerry$_{12} \bullet_{15}$ fails due to \cite[Section 5.2 (4)]{T&Jpart1}.

\paragraph{\textbf{\#11123, Jerry$_{24} \bullet_{25}$}}
The entry $a_{25}$ is already 0. In addition, $a_{15}$ can be made 0 by row/column operations. Thus, Jerry$_{24} \bullet_{25}$ fails due to \cite[Section 5.2 (3)]{T&Jpart1}.
\vspace{0.5em}

The Jerry formats above that do not satisfy all the conditions of Theorem \ref{criterion for Jerry} simultaneously: for all of them it is possible to find a reason for failure as in \cite[Section 5]{T&Jpart1}. 
All the other Jerry formats satisfying the conditions of Theorem \ref{criterion for Jerry} simultaneously do not present any reason for failure. Running the \verb|Magma| code implemented in \cite{T&Jpart1} and mentioned in the proof of Lemma \ref{from Z to Z2Z invariance} confirms this. This concludes the proof of Theorem \ref{criterion for Jerry}.

\begin{remark}
	The Fano 3-fold with GRDB ID \#24078 of Jerry type is the only family that does not follow Theorem \ref{criterion for Jerry}. Indeed, its codimension 3 source \#24077 in Jerry$_{12}$ format has even number of nodes and satisfies Condition \ref{condition c}, but the weight of its pivot is 1, and its basket is $\mathcal{B}_{Z} \coloneqq \{ \frac{1}{2}(1,1,1) \}$. 
	However, employing a code analogous to the ones in Section \ref{examples and codes}, it is still possible to produce explicit equations for its corresponding index 2 Fano \#40933.
	The ambient space of \#24077 is $\Proj^6(1^6,2)_{x_i,y}$, and suppose that $\gamma$ changes sign to the coordinate $x_1$, so $\xi \coloneqq x_1^2$. This Fano 3-fold is peculiar because it is the only index 1 double cover whose ambient space has enough coordinates of weight 1 to fill the entries of $M$ in in Jerry$_{12}$ format appropriately. 
	Suppose $I_D \coloneqq \langle x_3,x_4,x_5,x_6 \rangle$. The grading of $M$ is the following, and, for $q_1,q_2,q_3 \in I_D$ homogeneous polynomials of degree 1, 2, 2 respectively, it can be filled as
	\begin{center}
		\begin{tabular}{c c c}
			$\left(
			\begin{array}{c c c c}
			1 & 1 & 1 & 2 \\
			& 1 & 1 & 2 \\
			& & 1 & 2 \\
			& & & 2
			\end{array}
			\right) $
			& $\leadsto$
			&
			$\left(
			\begin{array}{c c c c}
			x_3 & x_4 & x_5 & q_1 \\
			& x_6 & q_2 & q_3 \\
			& & x_2 & y \\
			& & & \xi
			\end{array}
			\right) = M \; .$
		\end{tabular}
	\end{center}
	This shows that only $q_2$ could be made 0 after row/column operations on $M$, and that there is no further row/column elimination possible. Thus, \#24077 does not fall into the failure description of \cite[Section 5]{T&Jpart1}.
\end{remark}

\section{Proof of Main Theorem: Type I unprojections}

The proof of Theorem \ref{Main Theorem} relies on the possibility to realise $X$ in codimension 4 and index 1 as invariant under the action $\gamma$. 
The construction of $X$ starts from $Z$ in codimension 3. We gave necessary and sufficient conditions to realise $Z$ as invariant under $\gamma$ in Theorems \ref{criterion for Tom} and \ref{criterion for Jerry} for both types of deformation families of $Z$. 
It is possible to verify that the invariance under~$\gamma$ is maintained in $X$, that is, performing the unprojection does not affect the $\gamma$-invariance.

The equations of $X$ are of the form 
\begin{equation} \label{equations of X}
X = \{ \Pf_i(M) = s y_j - g_j = 0 \; | \; i=0,\dots,4 \text{ and } j=1,\dots,4 \}
\end{equation}
for $g_j=g_j(x_1,x_2,x_3,y_1,y_2,y_3,y_4)$ an homogeneous polynomial of degree $r+d_j$. We call $s y_j - g_j = 0$ the four \textit{unprojection equations} (cf \cite[Definition 1.2]{PapadakisReidKM}).  
In the following, we refer to \cite[Section 5]{PapadakisComplexes} and \cite[Appendix]{CampoSarkisov} for notation and definitions. 

First, let us consider the case of Tom formats and, to fix ideas, suppose that the matrix $M$ is in Tom$_1$ format (the other Tom formats are analogous). Then, $M$ is looks like
\begin{equation*} \label{general form T1} M =
\left(
\begin{array}{c c c c}
p_1 & p_2 & p_3 & p_4 \\
& a_{23} & a_{24} & a_{25} \\
& & a_{34} & a_{35} \\
& & & a_{45}
\end{array}
\right) 
\end{equation*}
where the $a_{ij} \in I_D$ are polynomials of the form $ a_{ij} := \sum_{k=1}^{4} \alpha_{ij}^k y_k $ for some polynomial coefficients~$\alpha_{ij}^k$, and $p_j \notin I_D$. Assume that all the entries of $M$ are $\gamma$-invariant. Here we calculate $\Pf_i$ by excluding the $(i+1)$-th row and the $(i+1)$-th column of $M$ for $i \in \{0,1,2,3,4\}$. 
For $i \not= 0$, we can define the matrix $ Q = \left( \Pf_i(N_j) \right)_{i,j=1 \dots 4}$ where 
\begin{equation*}
N_i =
\left(
\begin{array}{c c c c}
p_1 & p_2 & p_3 &  p_4 \\
& \alpha^i_{23} & \alpha^i_{24} & \alpha^i_{25} \\
& & \alpha^i_{34} & \alpha^i_{35} \\
& & & \alpha^i_{45}
\end{array}
\right) 
\end{equation*}
and $\alpha^i_{kl}$ is the coefficient of $y_i$ in $a_{kl}$. The polynomial $g_1$, and analogously the other $g_2,g_3,g_4$, is defined as 
\begin{equation*} \label{def of unproj in Pap}
g_1 = \frac{1}{p_1} \det \left(
\begin{array}{c c c}
\Pf_2(N_2) & \Pf_2(N_3) & \Pf_2(N_4) \\
\Pf_3(N_2) & \Pf_3(N_3) & \Pf_3(N_4) \\
\Pf_4(N_2) & \Pf_4(N_3) & \Pf_4(N_4)
\end{array}
\right) 
\end{equation*}
and the determinant on the right-hand side is divisible by $p_1$ (cf Lemma 5.3 of \cite{PapadakisComplexes}). 

Now, if $Z$ in Tom$_1$ format is $\gamma$ invariant then all the entries of $M$ are $\gamma$-invariant. Therefore, the same holds for the matrices of coefficients $N_i$. Also, it is clear that $\Pf_i(N_j)$ and the determinant in \eqref{def of unproj in Pap} are all sums and multiplications of $\gamma$-invariant polynomials. Thus, $X$ is $\gamma$-invariant, and it is possible to perform the $\Z/2\Z$ quotient of $X$ to obtain $\tilde{X}$. 

By Lemmas \ref{index of Xtilde}, \ref{Xtilde has terminal sings} and Theorem \ref{Xtilde in grdb}, we have just constructed a Tom type deformation family of a codimension 4 Fano 3-fold having Fano index 2.

For $M$ of Jerry format, we proceed analogously to the Tom case by retracing the construction of the unprojection equations in \cite[Section 5.7]{PapadakisComplexes}. 
In this case, to fix ideas suppose that $M$ is in Jerry$_{12}$ format, so it is of the form
\begin{equation*} \label{general form J12} M =
\left(
\begin{array}{c c c c}
c & a_1 & a_2 & a_3 \\
& b_1 & b_2 & b_3 \\
& & p_1 & p_2 \\
& & & p_3
\end{array}
\right) 
\end{equation*}
where $a_i = \sum_{k=1}^{4} \alpha_i^k y_k$, $b_i = \sum_{k=1}^{4} \beta_i^k y_k$, and $c = \sum_{k=1}^{4} \gamma^k y_k$ are all polynomials in $I_D$, and $p_i \notin I_D$ for $i=1,2,3$ (see \cite[(5.9), Section 5.7]{PapadakisComplexes}). Again, assume that $Z$ is $\gamma$-invariant, so each of the entries of $M$ are too. 

Define the $3 \times 4$ matrix $Q$ to have entries
\begin{align*}
Q_{1k} &\coloneqq \beta_1^k p_3 - \beta_2^k p_2 + \beta_3^k p_1 \\
Q_{2k} &\coloneqq \alpha_1^k p_3 - \alpha_2^k p_2 + \alpha_3^k p_1 \\
Q_{3k} &\coloneqq \gamma^k p_3 - \gamma^k p_2 + \gamma^k p_1 
\end{align*}
and define the polynomials $h_i$ as the determinant of the $3 \times 3$ minor of $Q$ after removing the $i$-th column. Observe that each $h_i$ is $\gamma$-invariant because it is a determinant of a $\gamma$-invariant matrix. 

By \cite[Lemma 5.11]{PapadakisComplexes} we have that for each $h_i$ there exist two polynomials $K_i, L_i$ such that $h_i = p_3 K_i + (a_2 p_2 - a_3 p_1) L_i$. These two polynomials are defined implicitly by considering the matrix $Q$ restricted to $p_3=0$. 
The polynomial $L_i$ is the determinant of the rightmost matrix appearing in the proof of \cite[Lemma 5.11]{PapadakisComplexes}, and it is $\gamma$-invariant because it is a determinant of a $\gamma$-invariant matrix; the polynomial $K_i$ is the remainder. Since $(a_2 p_2 - a_3 p_1), p_3, L_i, h_i$ are $\gamma$-invariant, so is $K_i$. 

By \cite[Equation (5.12) and Lemma 5.12]{PapadakisComplexes} we have that the polynomials $g_i = K_i + a_1 L_i$ are there right-hand side of the unprojection equations in \eqref{equations of X}; in particular, they are $\gamma$-invariant. 
Thus, as before, $X$ is $\gamma$-invariant, and we can obtain $\tilde{X}$ taking the quotient of $X$ by the action $\gamma$. Such $\tilde{X}$ is in GRDB by Lemmas \ref{index of Xtilde}, \ref{Xtilde has terminal sings} and Theorem \ref{Xtilde in grdb}.

The discussion about the two double cover Fano 3-folds arising from Type II$_2$ unprojections is contained in the next Section \ref{Type II}. In this case we explicitly build the $\gamma$-invariant equations of $\tilde{X}$, and the conclusion follows analogously as the Tom and Jerry cases.

\section{Proof of Main Theorem: Type II$_2$ unprojections}\label{Type II}

The two Hilbert series \#39569 and \#39607 have candidate double covers, \#512 and \#872 respectively, that do not have any Type I centre. Indeed, \#512 and \#872 are obtained as Type II$_2$ unprojections of the corresponding Fano hypersurfaces $Y$ \cite[Section 5.2]{Taylor}. 
The Tom and Jerry construction is not appropriate in this case, but different deformation families can still arise. For the theory of Type II unprojections we follow ideas contained in \cite{PapadakisTypeII1,PapadakisGeneralTheoryUnproj,PapadakisTypeII,PapadakisReidKM,Taylor} and in the paper in preparation \cite{TaylorTypeII}. In particular, in this Section we follow closely \cite{PapadakisTypeII} for the general strategy, and \cite{Taylor} for part of the explicit approach. The calculations presented in this Section
are supported by computer algebra and inspired by the above works; as for now, there is no general theory of Type II$_n$ unprojections. 
Here we show that Theorem \ref{Main Theorem} still holds for \#39569 and \#39607 by exhibiting explicit equations and by hand-checking the Gorenstein-ness of the varieties produced.

\subsection{\#39607} \label{872}

The index 1 codimension 4 Fano 3-fold \#872 $X \subset \Proj^7(1,3^2,4,5^2,6,7)$ is candidate to be a double cover of \#39607 of index 2. Its basket is $\mathcal{B}_{X} \coloneqq \{ 5\times \frac{1}{3}(1,1,2), \frac{1}{5}(1,1,4) \}$, and $\frac{1}{5}(1,1,4)$ is its Type II$_2$ centre. 
It comes from the Type II$_2$ unprojection of the hypersurface \#866 $Y_{15} \subset \Proj^4(1,3^2,4,5)$ having basket $\mathcal{B}_{Y} \coloneqq \{ 5\times \frac{1}{3}(1,1,2), \frac{1}{4}(1,1,3) \}$.
Following \cite{Taylor}, $X$ is obtained by unprojecting the divisor $D$ defined as the image of $\Proj^2(1,1,4)_{a,b,c}$ inside $Y$ via the embedding
\begin{align*}
\phi \colon \Proj^2(1,1,4) &\longrightarrow \Proj^4(1,3,3,4,5) \\ \nonumber
\left(a,b,c\right) &\longmapsto \left(a, b^3, b^3, c, b c \right)
\end{align*}

In this way, $D \coloneqq \Imago(\phi)$ is a divisor inside the general hypersurface $Y_{15} \subset \Proj^4(1,3^2,4,5)$ with homogeneous coordinates $x,u,z,y,v$. In addition, $\Imago(\phi)$ can be written as the $3\times 3$ minors of the following $3\times 6$ matrix (\cite[Example 5.2.2]{Taylor})
\begin{equation*}
N = \left(
\begin{array}{c c c c c c}
u & v & 0 & 0 & x^2 z & - y z \\
x^2 & -y & u & v & 0 & 0 \\
0 & 0 & x^2 & -y & u & v
\end{array}
\right) \; .
\end{equation*}

The equation defining $Y_{15}$ is an homogeneous polynomial of degree 15 of the form 
\begin{equation} \label{hypersurface eqn Type II_2}
F =  \sum_{i=1}^{20} A_i N_i
\end{equation}
where $N_i$ is the $i$-th $3\times 3$ minor of $N$, and $A_i$ is a general homogeneous polynomial of degree $15 - \deg(N_i)$. Note that here both $N_{15}$ and $N_{20}$ have already degree 15, so $A_{15}$ and $A_{20}$ are both non-zero constants. Since $x$ appears already as a square in $N$, we need that $x$ only appears with even powers in the polynomials $A_i$; in this way, $Y_{15}$ is invariant under $\gamma$.

For Type II$_2$ unprojections, there will be three new unprojection variables: indeed, the codimension increases by 3. We call them $s_1,s_2,s_3$, and they have weights $5,6,7$ respectively. 
Finding the equations for $X$ consists in finding the linear and quadratic relations between the unprojection variables $s_1,s_2,s_3$ (in the spirit of \cite[Subsection 2.2 and Lemma 2.5]{PapadakisTypeII}).

Consider the polynomial ring $R$ generated by the coordinates of $\Proj^4(1,3^2,4,5)$, the monomials $p_1 \coloneqq -a^2$, $p_2 \coloneqq c$, and the polynomials $A_i$ (cf \cite[Definition of $\mathcal{O}_{amb}$]{PapadakisTypeII}). The equations of $X$ are retrieved by looking at the resolution of $R/\langle F=0 \rangle$, and they consist in sums and products of the coordinates of $\Proj^4(1,3^2,4,5)$, the monomials $p_1 \coloneqq -a^2$, $p_2 \coloneqq c$, and the polynomials $A_i$. Thus, $x$ still retains its property of appearing only with even powers, and $X$ is defined by 9 equations invariant under $\gamma$. 
So, we can quotient $X$ by the action $\gamma$, and obtain $\tilde{X}$ analogously to Section \ref{double covers}. This constructs \#39607.

Note that it is possible to check with a \verb|Macaulay2| routine that the ring extension $R[s_1,s_2,s_3]$ quotient by the nine equations of $X$ is Gorenstein as follows. First insert the input data: field, polynomial ring \verb|R| (including the general polynomials $A_i$ and the unprojection variables $s_1,s_2,s_3$), matrix \verb|N|. Then, define \verb|J| to be the ideal generated by the $3 \times 3$ minors of $N$. The ideal of the hypersurface $Y_{15}$ is \verb|I|.
\begin{verbatim}
J = minors (3, N)
I = ideal (F= A0*J_(0) + A1*J_(1) + A2*J_(2) + A3*J_(3) + A4*J_(4) + A5*J_(5) 
+ A6*J_(6) + A7*J_(7) + A8*J_(8) + A9*J_(9) + A10*J_(10) + A11*J_(11) 
+ A12*J_(12) + A13*J_(13) + J_(14) + A15*J_(15) + A16*J_(16) + A17*J_(17) 
+ A18*J_(18) + J_(19) )
\end{verbatim}
To construct the unprojection ideal \verb|unprI|, we do
\begin{verbatim}
M = Hom (J, R^1/I)
tempI = ideal (matrix {{1,s1,s2,s3 }} * (presentation M)) + ideal(F)
unprI =  tempI : ideal(y)   
\end{verbatim}
To check that \verb|unprI| is Gorenstein, it is enough to use
\begin{verbatim}
betti res unprI
\end{verbatim}
and verify that its output gives a palindromic list of dimensions at the various stages of the resolution (part of the output has been omitted).
\begin{verbatim}
       0 1  2 3 4
total: 1 9 16 9 1
\end{verbatim} 
The routine is analogous for the example below in Subsection \ref{512}. 

We are not aware of any work in the literature regarding the question of whether there are multiple deformation families induced by Type II$_2$ unprojections, and if there are, how many. This procedure constructs at least one.

\subsection{\#39569} \label{512}
We proceed analogously for \#39569. Its double cover candidate is \#512 $X \subset \Proj^7(1,3,5,6,7^2,8,9)$, whose basket is $\mathcal{B}_{X} \coloneqq \{ 3\times \frac{1}{3}(1,1,2), \frac{1}{5}(1,2,3), \frac{1}{7}(1,1,6)  \}$, and $\frac{1}{7}(1,1,6)$ is one of its Type II$_2$ centres. It is induced by the Type II$_2$ unprojection of the hypersurface \#508 $Y_{21} \subset \Proj^4(1,3,5,6,7)$ with basket $\mathcal{B}_{Y} \coloneqq \{ 3\times \frac{1}{3}(1,1,2), \frac{1}{5}(1,2,3), \frac{1}{6}(1,1,5) \}$.
The divisor $D$ is defined as the image of $\Proj^2(1,1,6)_{a,b,c}$ inside $Y$ via the embedding
\begin{align*}
\phi \colon \Proj^2(1,1,6) &\longrightarrow \Proj^4(1,3,5,6,7) \\ \nonumber
\left(a,b,c\right) &\longmapsto \left(a, b^3, b^5, c, b c \right)
\end{align*}

The divisor $D \coloneqq \Imago(\phi)$ is given by the vanishing of the $3\times 3$ minors of the $3\times 6$ matrix 
\begin{equation*}
N = \left(
\begin{array}{c c c c c c}
u & v & -z^2 & 0 & 0 & - y z \\
0 & -y & u & v & -z^2 & 0 \\
-z & 0 & 0 & -y & u & v
\end{array}
\right) \; .
\end{equation*}
It sits inside the general hypersurface $Y_{21} \subset \Proj^4(1,3,5,6,7)$ with homogeneous coordinates $x,z,u,y,v$ respectively. The equation of $Y_{21}$ is again given by \eqref{hypersurface eqn Type II_2}, and we can impose that the general polynomials $A_i$ contain the variable $x$ only with even powers. As before, $A_{15}$ and $A_{20}$ are constants. 
The three new unprojection variables are $s_1,s_2,s_3$ with weights $7,8,9$ respectively. 
The unprojection of $Y$ and the quotient of $X$ by $\gamma$ constructs $\tilde{X}$ \#39569.

\section{Examples} \label{examples and codes}

Here we give two examples of our construction, one of a Tom family and one of a Jerry family. 
We explicitly construct the deformation families relative to the Hilbert series with GRDB ID \#39660. This is $\tilde{X} \subset \Proj^7(2,2,3,5,5,7,12,17)$, whose basket of singularities is $\mathcal{B}_{\tilde{X}} \coloneqq \{ \frac{1}{17}(2,5,12) \}$. 
The codimension 4 Fano 3-fold $X \subset \Proj^7(1,2,3,5,5,7,12,17)$ in index 1 with Hilbert series \#1158 is the candidate to be the double cover for \#39660. The coordinates of the ambient space of $X$ are $x,y,z,u_1,u_2,v,w,s$, and its basket of singularities is $\mathcal{B}_X \coloneqq \{ \frac{1}{17}(1,5,12) \}$. 
In this case, the grading of $M$ is
\begin{equation*}
\wts M = (m_{ij}) = \left(
\begin{array}{c c c c}
2 & 3 & 5 & 7 \\
& 5 & 7 & 9 \\
& & 8 & 10 \\
& & & 12
\end{array}
\right) \; .
\end{equation*}
The family of the codimension 3 Fano 3-fold $Z$ with GRDB ID \#1157 is composed of 3-folds sitting inside $\Proj^6(1,2,3,5,5,7,12)$ and whose equations are the five maximal pfaffians of $M$ with the above grading. 
The divisor $D$ is $\Proj^2(2,5,12)$. The Tom and Jerry formats of $M$ that admit an embedding $D \subset Z$ are Tom$_5$ and Jerry$_{12}$ each with 4 and 6 nodes on $D$ respectively. In addition, the pivot entry of the Jerry$_{12}$ format is $a_{12}$, which has weight 2. 
Therefore, by Lemmas \ref{nodes not fixed}, \ref{from Z to Z2Z invariance} we have two possible deformation families for $\tilde{X}$, one coming from the Tom$_5$ format, the other from the Jerry$_{12}$ format. It remains to exhibit the equations for these two families. These calculations can be checked using the \verb|tj| package for \verb|Magma| that can be found on the Graded Ring Database website \cite{grdb}. Instructions on how to fill the matrix $M$ can be found in \cite[Section 4]{T&Jpart1} and \cite[Section 3.2]{CampoSarkisov}.

\subsection{Finding equations for \#39660 of Tom$_5$ type}
Here we want to build $\tilde{Z}$ produced from $Z$ in Tom$_5$ format. We define the ambient space $\Proj^6(2,2,3,5,5,7,12)$ with coordinates as above, where $x$ has been replaced by $\xi$ of weight 2. 
The divisor $D \cong \Proj^2(2,5,12)$ is defined by the vanishing of the coordinates $y,z,u_2,v$. 
Here the matrix $M$ in Tom$_5$ format is filled following \cite[Section 6.2]{T&Jpart1}.
As a rule of thumb, we place homogeneous coordinates in the entries with matching degrees where possible according to the format. The rest of the entries are occupied by general polynomials in the given degrees, still maintaining the format. The matrix can be tidied up by row/column operations. 
For instance, we can fill the entries of $M$ as follows
\begin{equation*}
M = (a_{ij}) = \left(
\begin{array}{c c c c}
y & z & u_2 & -v \\
& -u_2 & v & - z^3 + \xi^2 u_1 \\
& & \xi^3 y + y^4 & -\xi^5 + u_1^2 \\
& & & w
\end{array}
\right) \; .
\end{equation*}
The equations of the 3-fold $\tilde{Z}$ are the maximal pfaffians of $M$, and it is possible to check with a \verb|Magma| routine that the number of nodes on $D$ is 2, in accordance to Lemma \ref{nodes not fixed}. 
The next step is to perform the Type I unprojection from $D \subset \tilde{Z}$ as in \cite{PapadakisReidKM}. The nine equations of $\tilde{X}$ are 
\begin{align*}
&\xi^3 y^2 + y^5 - u_2^2 - z v = 0 \\
&\xi^5 y - z^4 + \xi^2 z u_1 - y u_1^2 - u_2 v = 0\\
&z^3 u_2 - \xi^2 u_1 u_2 - v^2 + y w = 0 \\
&\xi^5 u_2 - \xi^3 y v - y^4 v - u_1^2 u_2 + z w = 0 \\
&\xi^3 y z^3 + y^4 z^3 - \xi^2 y^4 u_1 - \xi^5 v - z^4 u_1 + \xi^2 z u_1^2 - y u_1^3 +
u_1^2 v - u_1 u_2 v + u_2 w = 0 \\
&\xi^5 z^3 - \xi^7 u_1 - z^3 u_1^2 + \xi^2 u_1^3 - v w + y s = 0 \\
&-\xi^{10} + 2 \xi^5 u_1^2 + \xi^3 v^2 + y^3 v^2 - u_1^4 - z s = 0 \\
&\xi^3 z^3 v + y^3 z^3 v - \xi^5 u_1 v - \xi^2 y^3 u_1 v - \xi^5 w + u_1^2 w + u_2 s = 0 \\
&-\xi^8 y z^2 - \xi^5 y^4 z^2 + \xi^5 z^3 u_1 + \xi^2 y^3 z^3 u_1 - \xi^7 u_1^2 -
\xi^4 y^3 u_1^2 \\ &+ \xi^3 y z^2 u_1^2 + y^4 z^2 u_1^2 + \xi^3 z^2 u_2 v +
y^3 z^2 u_2 v + w^2 - v s = 0 \; .
\end{align*}

\subsection{Finding equations for \#39660 of Jerry$_{12}$ type}
In a similar fashion to the Tom case above, here the matrix $M$ in Jerry$_{12}$ is defined as 
\begin{equation*}
M = (a_{ij}) = \left(
\begin{array}{c c c c}
y & z & u_2 & v \\
& -u_2 & v & \xi^3 z + z^3 \\
& & \xi^4 + y^4 & \xi y^4 - u_1^2 \\
& & & w
\end{array}
\right) \; .
\end{equation*}
The codimension 3 Fano 3-fold $Z$ \#1157 satisfies the conditions required in Theorem \ref{criterion for Jerry}. In particular, $\tilde{Z}$ has 3 nodes on $D$. 
The equations of $\tilde{X}$ of Jerry$_{12}$ type with GRDB ID \#39660 are
\begin{align*}
&\xi^4 y + y^5 - u_2^2 - z v =0 \\
&\xi y^5 - \xi^3 z^2 - z^4 - y u_1^2 - u_2 v =0 \\
&\xi^3 z u_2 + z^3 u_2 - v^2 - y w =0 \\
&\xi y^4 u_2 - \xi^4 v - y^4 v - u_1^2 u_2 - z w = 0 \\
&\xi^7 z + \xi^3 y^4 z + \xi^4 z^3 + y^4 z^3 - \xi y^4 v + u_1^2 v - u_2 w =0 \\
&\xi^7 u_2 + \xi^4 z^2 u_2 - \xi^3 z u_1^2 - z^3 u_1^2 + v w + y s =0 \\
&\xi y^4 u_1^2 - \xi y^3 u_2 v + y^3 v^2 - \xi^4 w - u_1^4 - z s =0 \\
&\xi^{11} + \xi^7 y^4 + \xi^8 z^2 + \xi^4 y^4 z^2 + \xi^3 y^3 z v + y^3 z^3 v 
- \xi y^3 v^2 - u_1^2 w + u_2 s =0 \\
&\xi^7 u_1^2 + \xi^3 y^4 u_1^2 - \xi^4 y^3 u_2^2 + \xi^4 z^2 u_1^2 + 
y^4 z^2 u_1^2 - \xi y^3 z^2 u_2^2 + \xi^3 y^3 u_2 v + y^3 z^2 u_2 v + w^2 - v s =0 \; .
\end{align*}

\bibliography{bibliography}

\begin{thebibliography}{10}

\bibitem{AltinokBrownReidFanoK3}
S.~Alt{\i}nok, G.~Brown, and M.~Reid.
\newblock Fano 3-folds, {$K3$} surfaces and graded rings.
\newblock In {\em Topology and geometry: commemorating {SISTAG}}, volume 314 of
  {\em Contemp. Math.}, pages 25--53. Amer. Math. Soc., Providence, RI, 2002.

\bibitem{grdb}
G.~Brown, A.~M. Kasprzyk, et~al.
\newblock Graded {R}ing {D}atabase.
\newblock {\em Online. Access via \url{http://www. grdb. co. uk}}.

\bibitem{brownP2xP2}
G.~Brown, A.~M. Kasprzyk, and M.~I. Qureshi.
\newblock Fano 3-folds in {$\Bbb P^2\times \Bbb P^2$} format, {T}om and
  {J}erry.
\newblock {\em Eur. J. Math.}, 4(1):51--72, 2018.

\bibitem{T&Jpart1}
G.~Brown, M.~Kerber, and M.~Reid.
\newblock Fano 3-folds in codimension 4, {T}om and {J}erry. {P}art {I}.
\newblock {\em Compos. Math.}, 148(4):1171--1194, 2012.

\bibitem{TJBigTable}
G.~Brown, M.~Kerber, and M.~Reid.
\newblock Tom and {J}erry table, part of "{F}ano 3-folds in codimension 4,
  {T}om and {T}erry. {P}art {I}".
\newblock {\em Compositio Mathematica}, 148(4):1171--1194, 2012.

\bibitem{BrownSuzuki}
G.~Brown and K.~Suzuki.
\newblock Fano 3-folds with divisible anticanonical class.
\newblock {\em Manuscripta Math.}, 123(1):37--51, 2007.

\bibitem{CampoSarkisov}
L.~Campo.
\newblock Sarkisov links for index 1 fano 3-folds in codimension 4.
\newblock {\em arXiv preprint arXiv:2011.12209, to appear in Math. Nachr.},
  2020.

\bibitem{CoughlanDucat}
S.~Coughlan and T.~Ducat.
\newblock Constructing {F}ano 3-folds from cluster varieties of rank 2.
\newblock {\em Compos. Math.}, 156(9):1873--1914, 2020.

\bibitem{ducat2018alaprokhreid}
T.~Ducat.
\newblock Constructing {$\Bbb Q$}-{F}ano 3-folds \`a la {P}rokhorov \& {R}eid.
\newblock {\em Bull. Lond. Math. Soc.}, 50(3):420--434, 2018.

\bibitem{TakaoFujita}
T.~Fujita.
\newblock {\em Classification theories of polarized varieties}, volume 155 of
  {\em London Mathematical Society Lecture Note Series}.
\newblock Cambridge University Press, Cambridge, 1990.

\bibitem{Hartshorne}
R.~Hartshorne.
\newblock {\em Algebraic geometry}.
\newblock Springer-Verlag, New York-Heidelberg, 1977.
\newblock Graduate Texts in Mathematics, No. 52.

\bibitem{Iskovskih1}
V.~A. Iskovskih.
\newblock Fano threefolds. {I}.
\newblock {\em Izv. Akad. Nauk SSSR Ser. Mat.}, 41(3):516--562, 717, 1977.

\bibitem{Iskovskih2}
V.~A. Iskovskih.
\newblock Fano threefolds. {II}.
\newblock {\em Izv. Akad. Nauk SSSR Ser. Mat.}, 42(3):506--549, 1978.

\bibitem{KustinMiller}
A.~R. Kustin and M.~Miller.
\newblock Constructing big {G}orenstein ideals from small ones.
\newblock {\em J. Algebra}, 85(2):303--322, 1983.

\bibitem{PapadakisComplexes}
S.~A. Papadakis.
\newblock Kustin-{M}iller unprojection with complexes.
\newblock {\em J. Algebraic Geom.}, 13(2):249--268, 2004.

\bibitem{PapadakisTypeII}
S.~A. Papadakis.
\newblock Type {II} unprojection.
\newblock {\em J. Algebraic Geom.}, 15(3):399--414, 2006.

\bibitem{PapadakisGeneralTheoryUnproj}
S.~A. Papadakis.
\newblock Towards a general theory of unprojection.
\newblock {\em J. Math. Kyoto Univ.}, 47(3):579--598, 2007.

\bibitem{PapadakisTypeII1}
S.~A. Papadakis.
\newblock The equations of type {$II_1$} unprojection.
\newblock {\em J. Pure Appl. Algebra}, 212(10):2194--2208, 2008.

\bibitem{PapadakisReidKM}
S.~A. Papadakis and M.~Reid.
\newblock Kustin-{M}iller unprojection without complexes.
\newblock {\em J. Algebraic Geom.}, 13(3):563--577, 2004.

\bibitem{ProkhorovReid}
Y.~Prokhorov and M.~Reid.
\newblock On {$\Bbb Q$}-{F}ano 3-folds of {F}ano index 2.
\newblock In {\em Minimal models and extremal rays ({K}yoto, 2011)}, volume~70
  of {\em Adv. Stud. Pure Math.}, pages 397--420. Math. Soc. Japan, [Tokyo],
  2016.

\bibitem{SuzukiIndexBound}
K.~Suzuki.
\newblock On {F}ano indices of {$\Bbb Q$}-{F}ano 3-folds.
\newblock {\em Manuscripta Math.}, 114(2):229--246, 2004.

\bibitem{Taylor}
R.~Taylor.
\newblock {\em Type {II} unprojections, {F}ano threefolds and codimension four
  constructions}.
\newblock PhD thesis, University of Warwick, 2020.

\bibitem{TaylorTypeII}
R.~Taylor.
\newblock On {T}ype {II} unprojections.
\newblock {\em In preparation}, 2022.

\end{thebibliography}
\bibliographystyle{plain}

\end{document}